\documentclass[10pt,reqno]{amsproc}
\usepackage{amsmath}
\usepackage{amssymb}
\usepackage[all]{xy}
\usepackage{psfrag}
\usepackage{graphicx}
\usepackage{enumerate}
\usepackage{algorithmic}
\usepackage{algorithm}

\swapnumbers
\theoremstyle{plain}
\newtheorem{theorem}{Theorem}[section]

\newtheorem{lemma}[theorem]{Lemma}
\newtheorem{cor}[theorem]{Corollary}

\theoremstyle{definition}

\newtheorem{defn}[theorem]{Definition}

\newtheorem{remark}[theorem]{Remark}
\newtheorem{notation}[theorem]{Notation}

\newcommand{\norm}[2]{\left\| {#1}\right\| _{#2}}

\newcommand{\cl}[1]{\operatorname{cl}({#1})}

\newcommand{\I}{\mathcal I}
\newcommand{\C}{\mathbb C}

\newcommand{\N}{\mathbb N}
\newcommand{\T}{\mathbb T}

\newcommand{\tr}{\mathcal L}
\newcommand{\Hinf}[1]{H^\infty({#1})}

\usepackage[usenames, dvipsnames]{color}
\newcommand{\corr}[1]{{\textcolor{black}{#1}}}

\definecolor{mygray}{gray}{0.6}

\begin{document}

\date{3 April 2020}

\title[Lagrange approximation of transfer operators associated with holomorphic data]
{Lagrange approximation of transfer operators associated with holomorphic data}

\author[O.F.~Bandtlow]{Oscar F.~Bandtlow}

\address{%
Oscar F.~Bandtlow\\
School of Mathematical Sciences\\
Queen Mary University of London\\
London E3 4NS\\
UK.
}
\email{o.bandtlow@qmul.ac.uk}

\author[J.~Slipantschuk]{Julia Slipantschuk}

\address{%
Julia Slipantschuk\\
School of Mathematical Sciences\\
Queen Mary University of London\\
London E3 4NS\\
UK.
}
\email{j.slipantschuk@qmul.ac.uk}


\begin{abstract}
We show that spectral data of transfer operators given by holomorphic data
can be approximated using an effective numerical scheme based on Lagrange
interpolation. In particular, we show that for one-dimensional
systems satisfying certain complex contraction properties,
spectral data of the approximants converge exponentially to the
spectral data of the transfer operator with the exponential
rate determined by the respective complex contraction
ratios of the underlying systems.
We demonstrate the effectiveness of this scheme by numerically computing eigenvalues of transfer operators arising from interval and circle maps, as well as Lyapunov exponents
of (positive) random matrix products and iterated function systems, based on examples taken from the literature.
\end{abstract}

\maketitle

\section{Introduction}

Transfer operators constitute powerful tools for analysing the behaviour of hyperbolic dynamical systems, as their spectral data
encode various dynamic and geometric quantities of interest,
such as invariant measures,
exponential mixing rates, zeta functions counting periodic orbits, resonances of
hyperbolic surfaces, escape rates, stationary probability measures of iterated function
systems, or Lyapunov exponents of random matrix products
(see \cite{Ruelle, BoyarskiGora, Baladi1, Bahsoun, Borthwick,Baladi2}
for some pointers to the vast literature on this subject).

Unfortunately, spectral data of transfer operators, which typically act on
infinite-dimensional spaces, is
rarely available explicitly (see, however, \cite{BJS, SBJAnosov, FGL}
for examples with explicit spectra).
As a result, one needs to construct a suitable
discretisation of the operator in the form of
a matrix and numerically compute the
corresponding  eigenvalues and eigenfunctions. A popular approach
is to use a projection-based method, also known as `finite section method' or
`Galerkin method' in various other contexts: given a bounded operator
$L \colon B\to B$ on a Banach space $B$, one
considers an approximation scheme determined
by a sequence of rank-$n$ projections
$(\Pi_n\colon B\to B)_{n\in\mathbb{N}}$
converging to the identity in a suitable sense, and then takes spectral data of
$\Pi_n L \Pi_n$, represented by an $n\times n$ matrix,
as an approximation to spectral data of $L$.
However, unless extra conditions are
imposed on $(\Pi_n)_{n\in \mathbb{N}}$ and $L$, spectral data of
$(\Pi_n L \Pi_n)_{n\in \mathbb{N}}$
need not converge to spectral data of $L$, see, for example,
\cite{Hansen}.

The most widely studied projection-based method for the
spectral approximation of transfer
operators is the Ulam method, originally proposed in \cite{Ulam}
as a means to compute
the fixed point of the Perron-Frobenius operator for expanding
interval maps, which yields the density of the
unique absolutely continuous invariant measure for the underlying map.
The Ulam method is based on
partitioning the phase space of the underlying dynamical system into $n$
disjoint sets and defining $\Pi_n$ to be the conditional
expectation with respect to this
partition, that is, $\Pi_n$ is the projection onto
the subspace spanned by functions
which are piecewise constant on this partition.
The convergence of the Ulam method for piecewise monotonic and
expanding interval maps was established in \cite{Li}. The rate of convergence
was proven to be $O(\log(n)/n)$ in \cite{Keller82,BK},
which in turn was shown to be optimal in \cite{BoseMurray},
even for systems with higher regularity.

There exists by now a considerable body of literature concerned
with the application of Ulam-type methods to approximate the leading eigenvalue
and eigenfunction as well as subleading eigenvalues of transfer operators,
including those arising from higher-dimensional expanding or hyperbolic maps, see
\cite{F97, BK, DJ, BKL, F07, BB, FG-TQ, GN, CF1} to name but a few.
For dynamical system with higher regularity,
the speed of convergence of
projection-based methods can be improved by choosing projections onto subspaces spanned by
functions of higher smoothness, see \cite{Liverani} for a discussion
of a general strategy
or \cite{BH} for an approach using wavelets.
For practical applications, however, these
methods are often less suitable,
as they typically involve numerical evaluation of integrals
as well as higher derivatives.

In this article, we shall study a spectral approximation scheme based
on interpolating
projections,
which, for transfer operators associated with holomorphic data,
yields exponential convergence of spectral data,
while remaining practically efficient.
In this scheme,
the projection operator $\Pi_n$ maps a function $f$ to the unique
(Laurent) polynomial of degree~$n$
that coincides with $f$ on a set of $n$ predefined
collocation points.
The resulting approximation scheme, variously known in other
contexts as Lagrange interpolation, spectral Galerkin or collocation method, is
easily implementable in practice, as the $n$-th approximant $\Pi_n\mathcal{L}\Pi_n$ of
the transfer operator $\mathcal{L}$ can be obtained from an $n\times n$
matrix of the form $\left ( (\mathcal{L}e_j)(x_i) \right )_{ij}$,
where $(e_j)_j$ is a suitable collection of (Laurent) polynomials and $(x_i)_i$ a
suitable collection of collocation points (both of which may depend on $n$).
This method has recently been applied
to transfer operators for expanding circle or interval maps
by Wormell \cite{Wormell}; see also \cite{Schottky}, where this method has been used
to effectively calculate resonances of Schottky surfaces using transfer operator methods. The analysis in \cite{Wormell} is
in fact based on a different projection scheme,
in which the rank-$n$ projection is chosen
to be the orthogonal projection
(with respect to a suitably weighted $L^2$-scalar product)
onto the first $n$ elements of a Fourier basis (for circle maps)
or a Chebyshev basis (for interval maps),
an idea that is already present in an earlier paper by
MacKernan and Basios~\cite{MacKernan}.
The main result of \cite{Wormell} shows that the
leading eigenfunction of the transfer
operator can be approximated exponentially fast
in bounded variation norm by the
leading eigenfunctions of the transfer operators
truncated using the above orthogonal
projection scheme.
The paper also provides an algorithm in which the $L^2$-inner
products arising in the matrices representing
the truncated transfer operator are
evaluated using Chebyshev--Gauss quadrature,
which effectively renders the described
algorithm into a Lagrange interpolation method,
where the resulting errors are controlled
using certain aliasing bounds.

By contrast, we directly investigate the (non-orthogonal)
Lagrange interpolation projection,
which allows us to obtain
uniform convergence of the approximation scheme on suitable Banach spaces
of holomorphic functions, which in turn yields
strong convergence results for all spectral data,
including eigenvalues as well as the corresponding generalised eigenfunctions and eigendistributions.
Moreover, we obtain explicit bounds on the exponential convergence
rate given by certain complex contraction ratios of the underlying system.
Whereas \cite{MacKernan, Wormell} are chiefly concerned
with Perron-Frobenius operators for expanding (Markov) maps,
our results apply to a more general class of
transfer operators associated to so-called
holomorphic map-weight systems (see, for example, \cite{BJ_LMS, BJ_Advances}).

We should also mention that there is a completely different approach for
approximating statistical properties of dynamical systems based on
an intimate relation between certain Fredholm determinants encoding eigenvalues of
transfer operators and periodic orbits of the underlying systems.
This method became popular with the papers \cite{AAC1,AAC2} and was
investigated rigorously for approximating invariant densities and metric entropy
of analytic expanding Markov maps \cite{JP00}, computing
Lyapunov exponents of random matrix products of positive matrices \cite{P2010,JM},
as well as other dynamical quantities and
invariants~\cite{JPAJM, JPAdvances, BanJP, JP4, JPV, CJ09}.
The convergence of these periodic orbit based algorithms is super-exponential
in the maximum period of periodic orbits used. However,
the number of periodic orbits grows exponentially with the period,
rendering their computation intractable for large periods.
Moreover, the expression involving periodic orbits used in these algorithms can become
numerically unstable for transfer operators with large weights, see the discussion in \cite{Schottky}.

The structure of the paper is as follows.
We begin with a brief introduction to the Lagrange interpolation problem on
the unit circle (Section \ref{sec:circle_interpolation}).
We then specialise to the case of equidistant interpolation, and,
in Section \ref{sec:equid}, prove a key lemma
for the error of approximation of the interpolation operator when
considered on Hardy spaces of bounded holomorphic functions
on certain annuli containing
the unit circle. In Section \ref{sec:lag_cheb}, we review the interpolation problem on the
interval $[-1,1]$ with the interpolation points chosen as zeros of Chebyshev polynomials,
and show that the corresponding Lagrange--Chebyshev interpolation operator
considered on the Hardy space of bounded holomorphic functions on Bernstein ellipses is isometrically isomorphic
to the equidistant interpolation operator from Section \ref{sec:equid}, yielding the same
bounds for the error of approximation. In Section \ref{sec:lag_cheb_gen} the results from
the previous section are extended to Hardy spaces of bounded holomorphic functions
on more general (confocal) ellipses.
In Section \ref{sec:trans_ellipse}, we study transfer operators associated to holomorphic
map-weight systems, and show (Theorem \ref{thm:ellipse}) that these can be
approximated exponentially fast in operator norm by finite-rank
Lagrange--Chebyshev approximants. An analogous result  for generalised transfer
operators arising from analytic expanding circle maps (Theorem~\ref{thm:circle}) is
presented in Section \ref{trans_annulus}. The resulting exponential convergence of
spectral data, obtained as a straightforward consequence of our main results,
Theorem~\ref{thm:ellipse} and Theorem~\ref{thm:circle},
are collected as Corollaries~\ref{cor:conv} and \ref{cor:conv2} in Section~\ref{sec:conv}.
Finally, in Section \ref{sec:applications}, we demonstrate
that a practical algorithm based on Lagrange--Chebyshev approximation can be used to
effectively compute spectral data of suitable transfer operators, by applying it to several
examples from the literature and comparing it to other approximation methods. These
include eigenvalues of transfer operators arising from interval or circle maps, as well as
Lyapunov exponents for random matrix products and  iterated functions systems.

\section{Lagrange interpolation}
\subsection{Lagrange interpolation on the unit circle}
\label{sec:circle_interpolation}
We start by recalling basic facts concerning the Lagrange interpolation
problem on the unit circle (see, for example, \cite{DG-V}). For a brief
overview over interpolation in general, see, for example, \cite[Chapter~4]{Rivlin}.

Let $\T=\{z\in \C: |z|=1\}$ be the unit circle, and let $f\colon \T\to \C$ be a continuous
function. For $n\in \mathbb{Z}$ let $e_n$ denote the Laurent monomial $e_n(z) = z^n$.
We are interested in approximating $f$ by Laurent polynomials, that is, finite linear
combinations of Laurent monomials.

For $N\in\N$, let $\Gamma = \{z_k \colon k=0,\ldots, N-1\}\subset \T$ be a set
of $N$ distinct complex numbers on the unit circle, and let
$N_l$ and $N_u$ be two
non-negative integers with $N_l + N_u = N-1$.
We shall refer to the points in $\Gamma$ as interpolation nodes, or simply nodes.
The Lagrange interpolation problem in the space of Laurent polynomials
$\Lambda_{-N_l, N_u}=
\text{span} \{e_n\colon -N_l \leq n \leq N_u\}$ amounts to determining
the unique Laurent polynomial
$q_N \in \Lambda_{-N_l, N_u}$
with
\[q_N(z_k) = f(z_k)\, \text{ for } k=0,\ldots,N-1.\]
It is not difficult to see that the interpolant $q_N$
can be written as a linear combination of Lagrange Laurent polynomials as follows
\begin{equation}\label{eq:lagrange_sum}
q_N(z) = \sum_{k=0}^{N-1} f(z_k) l_k(z).
\end{equation}
Here, $l_k$ is the unique Laurent polynomial in $\Lambda_{-N_l, N_u}$
satisfying $l_k(z_j) = \delta_{jk}$, which is given by
\[ l_k(z) = \frac{l(z) z_k^{N_l}}{l'(z_k)(z-z_k) z^{N_l}},\]
where $l(z) = \prod_{k=0}^{N-1} (z - z_k)$.
Note that \eqref{eq:lagrange_sum} defines a projection operator on the space
of continuous functions, denoted by $Q_N\colon C(\mathbb{T})\to C(\mathbb{T})$
with $q_N = Q_{N}f$.

As we are interested in applications to transfer operators arising from analytic maps, we
shall study the operator $Q_N$ when acting on functions analytic on $\T$, which
have analytic extensions to certain domains containing $\T$.
With slight abuse
of notation we continue to write $f$ for various extensions as well.
An important ingredient for the proofs presented in the following subsections
is the observation that the sum in \eqref{eq:lagrange_sum} can be rewritten as a contour
integral over such domains. More precisely, assuming that $f$ is holomorphic
on the closure of an annulus $A$ containing $\mathbb{T}$, the Residue
Theorem implies that

\begin{equation}\label{eq:fQNf}
f(z) - Q_Nf(z) = \frac{1}{2\pi i} \int_{\partial A}
\frac{f(\zeta) l(z) \zeta^{N_l}}{(\zeta-z)l(\zeta) z^{N_l}}\, d\zeta
\qquad(z\in A),
\end{equation}

\begin{equation}\label{eq:QNf}
Q_Nf(z) = \frac{1}{2\pi i} \int_{\partial A}
\frac{f(\zeta)\left(l(\zeta) z^{N_l} - l(z) \zeta^{N_l}\right)}{(\zeta-z)l(\zeta) z^{N_l}}\, d\zeta
\qquad (z\in A),
\end{equation}
where $\partial A$ denotes the positively oriented boundary of $A$.
In order to see this, first note that $\frac{z_k^{N_l}}{l'(z_k) (z-z_k)}$
is the residue of $\zeta \mapsto \frac{\zeta^{N_l}}{l(\zeta)(z - \zeta)}$
at the simple pole $\zeta=z_k$ for any $z\in A$ with $z\neq z_k$.
For any such $z$, the Residue Theorem allows us to replace the sum in
\eqref{eq:lagrange_sum} by a contour integral
\[ q_N(z) =
\frac{1}{2\pi i} \int_{\mathcal{C}}
\frac{f(\zeta) l(z) \zeta^{N_l}} {(\zeta-z)l(\zeta) z^{N_l}} \,d\zeta,\]
where $\mathcal{C}$ is a simple closed positively oriented contour in $A$ enclosing all
$z_k$, $k=0,\ldots, N-1$ but not $z$. Enlarging
the contour of integration and accounting for the residue at $z$,
Equation \eqref{eq:fQNf} follows, which in turn implies Equation \eqref{eq:QNf}.
The expression in \eqref{eq:QNf} is also known as the Hermite integral
formula (see, for example, \cite[Theorem 11.1]{Trefethen}).

\subsection{Equidistant interpolation on the unit circle}\label{sec:equid}
We shall now specialise the Lagrange interpolation problem to nodes
equally spaced on $\T$.
For this, we set $N=2 n$, fix $N_l = n$, $N_u = n-1$, and choose
the $2n$ equidistant nodes as roots of $-1$, that is,
\begin{equation}\label{eq:zk}
z_k = \exp \left ( \frac{2k+1}{2n} i\pi  \right )  \qquad (k=0,\ldots,2n-1),
\end{equation}
so that $l(z) = z^{2n}+1$.
We may now rewrite the Laurent polynomial $l_k$ in
\eqref{eq:lagrange_sum} as follows
\[l_k(z) =
\frac{z^n + z^{-n}}{2n z_k^{n-1}(z-z_k)} =
\frac{1}{2n} \sum_{l=-n}^{n-1}\left( \frac{z}{z_k} \right)^{l}.\]
Now, using the above expression for $l_k$ and rearranging the sum in
\eqref{eq:lagrange_sum} we obtain the following
representation of the equidistant
interpolation projection $Q_N = Q_{2n}$
\begin{align}\label{eq:Q2nf_practical}
(Q_{2n}f)(z)
&= \frac{1}{2n}\sum_{l=-n}^{n-1} c_{l,2n}(f) e_l(z) \qquad (z\in \T, f\in C(\T) ),
\end{align}
where the coefficient functionals
\[c_{l, 2n}(f) =\sum_{k=0}^{2n-1} f(z_k) z_k^{-l} \qquad (f\in C(\T))\]
turn out to be the discrete Fourier transform of the sequence
$f(z_0), f(z_1), \ldots, f(z_{2n-1})$. As a result, the expression
$\eqref{eq:Q2nf_practical}$ is particularly useful for numerical implementation.
In the case where $f$ is not only continuous but also holomorphic
on the closure of an annulus $A$ containing $\T$, the representations
\eqref{eq:fQNf} and \eqref{eq:QNf} take the following form

\begin{equation}\label{eq:fQ2nf}
f(z) - (Q_{2n} f)(z) = \frac{1}{2\pi i} \int_{\partial A}
\frac{f(\zeta) \sigma(z^n) }{(\zeta-z)\sigma(\zeta^n)}\,d\zeta
\qquad (z\in A),
\end{equation}

\begin{equation}\label{eq:Q2nf}
(Q_{2n} f)(z) = \frac{1}{2\pi i} \int_{\partial A}
\frac{f(\zeta)\left(\sigma(\zeta^n) - \sigma(z^n)\right)}{(\zeta-z)\sigma(\zeta^n)}\,d\zeta
\qquad (z \in A),
\end{equation}
where we have used $l(z) / z^n = (z^n + z^{-n}) =  2 \sigma(z^n)$ with
$\sigma$ denoting the Joukowski transform
\begin{equation}\label{eq:sigma}
\sigma(z) = \frac{1}{2}(z + z^{-1}).
\end{equation}

We shall now prove a simple lemma that will turn out to be the key estimate allowing us
to establish convergence of the Lagrange interpolation method for transfer operators
arising from expanding analytic circle maps.  In order to formulate it we require some
more notation.

For $\rho>1$ let $A_\rho = \{z\in \C \colon \rho^{-1} < |z| < \rho\}$
be an open annulus with radii $\rho$ and $\rho^{-1}$.
We write $H^\infty(A_\rho)$ for the Hardy space of bounded holomorphic functions on
$A_\rho$ which, equipped with the norm
$\norm{f}{\corr{H^\infty(A_\rho)}} = \text{sup}\{|f(z)|\colon z \in A_\rho\}$,
 is a Banach space.

Observe now that for $1<r<R$, we have $A_r\subset A_R$, so $H^\infty(A_R)$ can be
identified with a subspace of $H^\infty(A_r)$ via the canonical embedding
$J \colon H^\infty(A_R) \to H^\infty(A_r)$ given by $Jf=f|_{A_r}$. The following lemma shows
that $J$ is approximated at exponential speed by the equidistant Lagrange
interpolation projections.

\begin{lemma}\label{lem:J}
For $1<r<R$
let $J\colon H^\infty(A_R) \to H^\infty(A_r)$ denote the canonical
embedding and consider $Q_{2n}$ as an operator
from $H^\infty(A_R)$ to $H^\infty(A_r)$.
Then, for any $n\in \N$,
\begin{equation}\label{eq:normJQ2n}
  \norm{J-Q_{2n}}{H^\infty(A_R)\to H^\infty(A_r)} \leq
c_{r,R}\,\frac{\cosh(n\log(r))}{\sinh(n\log(R))},
\end{equation}

\begin{equation}\label{eq:normQ2n}
\norm{Q_{2n}}{H^\infty(A_R)\to H^\infty(A_r)}\leq
c_{r,R}\,\frac{\cosh(n \log(R)) + \cosh(n\log(r))}{\sinh(n\log(R))},
\end{equation}
where
\[ c_{r,R} = \frac{\sinh(\log{R})}{\cosh(\log(R)) - \cosh(\log(r))}. \]
\end{lemma}

\begin{proof}
We start with the following simple inequalities:
\begin{enumerate}[(i)]
\item $|\sigma(z^n)|\leq \frac{1}{2}(r^n+r^{-n})$ for any $z\in A_r$ and any $n\in
\mathbb{N}$;
\item $|\sigma(\zeta^n)|\geq \frac{1}{2}(R^n-R^{-n})$ for any $\zeta \in \partial A_R$
and any $n\in \mathbb{N}$.
\end{enumerate}
Next, fix $f\in H^\infty(A_R)$ with $\norm{f}{H^\infty(A_R)} \leq 1$ and let
$z\in A_r$.
By \cite[Theorem 17.11]{Rudin}, the (nontangential) limit $f^*(\zeta)$ for
$\zeta \in \partial A_R$
exists a.e.\ and $f^*$ is integrable on $\partial A_R$. Moreover, the integral
representations (\ref{eq:fQ2nf}) and (\ref{eq:Q2nf}) remain valid with $\partial A_R$ in
place of $\partial A$.
Using the inequalities (i) and (ii) above, we thus obtain
\begin{multline*}
|f(z) - Q_{2n}f(z)| \leq  \frac {1}{2\pi}
\int_{|\zeta|=R} \frac{|f^\ast(\zeta)||\sigma(z^n)|}{|z - \zeta||\sigma(\zeta^n)|}\,|d\zeta|
+ \frac {1}{2\pi}
\int_{|\zeta|=R^{-1}} \frac{|f^\ast(\zeta)||\sigma(z^n)|}{|z - \zeta||\sigma(\zeta^n)|}
\,|d\zeta| \\
\leq
\left( \frac{R}{R-|z|}  + \frac{R^{-1}}{|z| - R^{-1}}
\right)
\left(\frac{r^n + r^{-n}}{R^n-R^{-n}}\right)
=\left(\frac{R-R^{-1}}{(R+R^{-1})-(|z|+|z|^{-1})}\right)
\left(\frac{r^n + r^{-n}}{R^n-R^{-n}}\right),
\end{multline*}
which, after invoking inequality (i) once more, yields
\[
|f(z) - Q_{2n}f(z)| \leq
\frac{\sinh(\log{R})}{\cosh(\log(R)) - \cosh(\log(r))}
\frac{\cosh(n\log(r))}{\sinh(n\log(R))},
\]
which in turn furnishes \eqref{eq:normJQ2n}.

Similarly, for $f\in H^\infty(A_R)$ with $\norm{f}{H^\infty(A_R)} \leq 1$ and
$z\in A_r$ we obtain
\begin{align*}
|Q_{2n}f(z)| \leq
\left(\frac{R-R^{-1}}{(R+R^{-1})-(r+r^{-1})}\right)
\left(\frac{R^n+R^{-n} + r^n + r^{-n}}{R^n-R^{-n}}\right)\,,
\end{align*}
from which \eqref{eq:normQ2n} follows.
\end{proof}

\begin{remark}
In particular, Lemma \ref{lem:J} implies that
\[\norm{J-Q_{2n}}{H^\infty(A_R)\to H^\infty(A_r)} =
O\left(\left({\frac{r}{R}}\right)^n\right) \text{ as } n\to \infty,\]
 \[\norm{Q_{2n}}{H^\infty(A_R)\to H^\infty(A_r)} =
O\left( 1 \right) \text{ as } n\to \infty,\]
that is, equidistant Lagrange interpolation is stable and converges exponentially to the canonical embedding of $H^\infty(A_R)$ in $H^\infty(A_r)$.
\end{remark}

\subsection{Lagrange--Chebyshev interpolation on the interval $[-1,1]$}
\label{sec:lag_cheb}
Lagrange interpolation on the unit circle at equidistant nodes is closely
related to Lagrange--Chebyshev
interpolation on the interval $[-1,1]$, as we shall see presently.

We write $T_n$ for the Chebyshev polynomial
of the first kind of degree $n$, which is  given by
$T_n(\cos(\theta)) = \cos (n\theta)$ for $n\in\mathbb{N}_0$.
The zeros of $T_n$, referred to as Chebyshev nodes (of order $n$),
are the orthogonal projections of the $2n$ equidistant nodes
in \eqref{eq:zk} onto the interval $[-1,1]$, and
are given by
\begin{equation}\label{eq:xk}
x_k = \cos\left ( {\frac{(2k+1)\pi }{2n}} \right ) \quad (k=0,\ldots, n-1).
\end{equation}

Suppose now that we are given a continuous function
$f\colon[-1,1] \to \mathbb{C}$. The Lagrange--Chebyshev interpolation
problem is to find the unique polynomial $P_nf$ of degree $n-1$ that
coincides with $f$ at the Chebyshev nodes ${x_0, \ldots, x_{n-1}}$
of order $n$. This polynomial
can be written
\[(P_nf)(x) = \sum_{k=0}^{n-1} f(x_k) l_k(x),\]
where $l_k$ is the Lagrange polynomial
associated with $x_k$,
given by
\[ l_k(x) = \frac{T_n(x)}{T_n'(x)(x-x_k)}.\]
The resulting projection operator $P_n \colon C([-1,1])\to C([-1,1])$ will be
referred to as
Lagrange--Chebyshev projection operator.
Following the same arguments as in Section \ref{sec:circle_interpolation},
for $f$ extending holomorphically to a complex neighborhood
$U\supset[-1,1]$, it is easy to see that $P_n$ has the representation
\begin{equation}\label{eq:Pnf}
(P_nf)(z) = \frac{1}{2\pi i}\int_{
  \mathcal{C}}\frac{f(\zeta)(T_n(\zeta)-T_n(z))}{(\zeta -z)T_n(\zeta)}\,d\zeta \qquad(z \in U),
\end{equation}
where $\mathcal{C}$ is a simple closed positively oriented contour in $U$
containing $z$ and the interval $[-1,1]$ in its interior.

The operators $P_{n}$ and $Q_{2n}$ are
intimately related, and, as we shall see in Lemma \ref{lem:Jhat},
the convergence properties of the former can be deduced from those
of the latter.
Before establishing this, we require some more
terminology.

\begin{notation} For $\rho\in \mathbb{R}$ with $\rho>1$ we write
\[E_\rho := \left\{ \frac{1}{2} (w + w^{-1}) : w \in A_\rho\right\}\] for
the domain in $\mathbb{C}$ containing the origin, bounded by the ellipse
with lengths of major and minor semi-axis
given by $a = \cosh(\log \rho)$ and $b = \sinh(\log \rho)$,
respectively, and foci at $-1$ and $1$. We shall refer to $E_\rho$ as a \emph{standard elliptic domain}, or, in slight abuse of terminology, simply as a \emph{standard ellipse}.
Note also that we have
\[E_\rho = \sigma(A_\rho).\]
\end{notation}

To each $E_\rho$ we associate the
Hardy space $H^\infty(E_\rho)$, that is, the Banach space
of bounded holomorphic functions on $E_\rho$ equipped with the norm
$\norm{f}{H^\infty(E_\rho)} = \text{sup}\{|f(z)|\colon z \in
E_\rho\}$.
We shall now
show that there is an
isometric isomorphism between $H^\infty(E_\rho)$ and the Banach space
\[H^{\infty}_e(A_\rho) = \{f\in H^\infty(A_\rho) \colon f(z) = f(1/z), \,\forall z\in A_\rho\}.\]

\begin{lemma}\label{lem:iso}
Let $D_\sigma$ be the
composition operator given by
\[D_\sigma\colon  f \mapsto f \circ \sigma.\]
Then, $D_\sigma$ is an isometric isomorphism between  $H^\infty(E_\rho)$ and
$H^\infty_e(A_\rho)$.
\end{lemma}
\begin{proof}
As $\sigma$ is holomorphic on $A_\rho$ with
$\sigma(A_\rho)\subseteq E_\rho$, it follows that $D_\sigma f$
is holomorphic on $A_\rho$
for any $f\in H^\infty(E_\rho)$.
Since $\sigma(z) = \sigma(1/z)$ it follows that
$(D_\sigma f)(z) = (D_\sigma f) (1/z)$.
Note that $\sigma$ is surjective (more precisely it is a two-to-one map from
$A_\rho$ to $E_\rho$), which implies that $D_\sigma$ is an isometry,
since
\[||D_\sigma f||_{H_e^\infty(A_\rho)}
= \sup_{z\in A_\rho} |(f\circ \sigma)(z)| =
\sup_{w \in \sigma(A_\rho)}| f(w)| =
\sup_{w\in E_\rho}|f(w)| = \|f\|_{H^\infty(E_\rho)}.\]
Moreover, the surjectivity of $\sigma$
implies that $D_\rho$ is injective.

In order to show that $D_\sigma$ is surjective,
we pick $g\in H_e^\infty(A)$ and write it as a Laurent series
$g(z) = \sum_{n=-\infty}^{\infty} c_n e_n(z)$.
Since $g(z) = g(1/z)$ it follows that
$c_n(g) = c_{-n}(g)$ for all $n\in\N$,
and so
\[g(z) = c_0(g) + 2\sum_{n=1}^{\infty} c_n(g) (\sigma(e_n(z))
= c_0(g) + 2\sum_{n=1}^{\infty} c_n(g) T_n (\sigma(z)).\]
Now, as this series converges uniformly on
compact subsets of $A_\rho$,
the surjectivity of
$\sigma$ implies uniform convergence of
$c_0(g) + 2\sum_{n=1}^{\infty} c_n(g) T_n(w)$
on compact subsets of $E_\rho$,
implying that $f(w) = c_0(g) + 2\sum_{n=1}^{\infty} c_n(g) T_n(w)$
is holomorphic on $E_\rho$.
As $\sup_{w\in E_\rho}|f(w)| = \|g\|_{H_e^\infty(A_\rho)}$ we have an
$f\in H^\infty(E_\rho)$ satisfying $g = D_\sigma f$.
\end{proof}

Using the fact that $\sigma(z) = \sigma(1/z)$ for any $z\in A_\rho$, we can consider
the operator $Q_{2n}$ given in \eqref{eq:Q2nf} as an operator from
$H^\infty_e(A_R)$ to $H^\infty_e(A_r)$ for some $1<r<R$.

\begin{lemma}\label{lem:Pn_def}
For $1 < r < R$, consider $Q_{2n}$ in \eqref{eq:Q2nf}
as an operator from $H^\infty_e(A_R)$ to $H^\infty_e(A_r)$, and consider
$P_n$ in \eqref{eq:Pnf} as an operator from
$H^\infty(E_R)$ to $H^\infty(E_r)$.
Denote by $D_{\sigma,q}\colon H^\infty(E_q) \to H_e^\infty(A_q)$ with
$q = r, R$,  the isomorphisms defined in Lemma \ref{lem:iso}.
Then,
 \[Q_{2n} D_{\sigma,R} = D_{\sigma, r} P_n.\]
 \end{lemma}

\begin{proof}
First observe that we have the following relation
\begin{equation}
\label{eq:keyrelation}
\frac{\zeta}{\zeta-z} - \frac{\zeta^{-1}}{\zeta^{-1}-z} =
\frac{1}{1-z\zeta^{-1}} + \frac{z^{-1}}{\zeta-z^{-1}} =
\frac{\zeta - \zeta^{-1}}{(1-z\zeta^{-1})(\zeta-z^{-1})} =
\frac{\sigma'(\zeta) \zeta}{\sigma(\zeta) - \sigma(z)},
\end{equation}
where $\zeta$ and $z$ are any non-zero complex numbers.
For any such $z$ we also have the relation
$T_n(\sigma(z)) = \frac{1}{2}(z^n + z^{-n}) = \sigma(e_n(z)) = \sigma(e_{-n}(z))$.
Using the shorthand
\[ K(\zeta, z) =
\frac{\sigma(\zeta^n) - \sigma(z^n)}{\sigma(\zeta^n)}
= \frac{T_n(\sigma(\zeta)) -
    T_n(\sigma(z))}{T_n(\sigma(\zeta))},
\]
we obtain, for any $f\in H^\infty(E_R)$, any $z\in A_r$ and any $\rho$ with $r<\rho < R$
\begin{align*}
Q_{2n} (D_{\sigma,R} f) (z)
&= \frac{1}{2\pi i} \int_{\partial A_\rho}
\frac{f(\sigma(\zeta))K(\zeta, z)}{\zeta-z}\, d\zeta
\\
&= \frac{1}{2\pi} \int_{0}^{2\pi}
\frac{f(\sigma(\rho e^{it})) K(\rho e^{it}, z)
\rho e^{it} }{\rho e^{it}-z} \,dt
-
\frac{1}{2\pi}\int_{0}^{2\pi}
\frac{f(\sigma((\rho e^{it})^{-1}))K((\rho e^{it})^{-1}, z)}{((\rho e^{it})^{-1}-z)\rho e^{it}} \,dt\\
&=
\frac{1}{2\pi} \int_{0}^{2\pi}
\frac{f(\sigma(\rho e^{it})) K(\rho e^{it}, z) \sigma'(\rho e^{it})\rho e^{it}}
{\sigma(\rho e^{it})-\sigma(z)} \,dt\\
&= \frac{1}{2\pi i} \int_{\partial E_\rho}
\frac{f(\zeta)(T_n(\zeta) - T_n(\sigma(z)))}{(\zeta-\sigma(z))T_n(\zeta)}\,d\zeta\\
&= D_{\sigma,r} (P_nf) (z),
\end{align*}
where the third equality uses relation \eqref{eq:keyrelation},
and the penultimate equality follows
from a change of variables with $\zeta = \sigma(\rho e^{it})$.
\end{proof}

The convergence properties of $P_n$ now follow from
the convergence properties of $Q_{2n}$.

\begin{lemma}\label{lem:Jhat}
Let $1<r<R$ and let
$\hat{J}\colon H^\infty(E_R) \to H^{\infty}(E_r)$ denote the canonical
embedding. Then, for any $n\in \N$,
\begin{equation}\label{eq:normJP2n}
  \|\hat{J} - P_n\|_{H^\infty(E_R)\to H^\infty(E_r)}
  \leq
c_{r,R}\, \frac{\cosh(n\log(r))}{\sinh(n\log(R))},
\end{equation}

\begin{equation}\label{eq:normP2n}
\norm{P_n}{H^\infty(E_R)\to H^\infty(E_r)}\leq
c_{r,R}\, \frac{\cosh(n \log(R)) + \cosh(n\log(r))}{\sinh(n\log(R))},
\end{equation}
where
\[ c_{r,R} = \frac{\sinh(\log{R})}{\cosh(\log(R)) - \cosh(\log(r))}. \]
\end{lemma}

\begin{proof}
 Let $D_{\sigma,r}$ and $D_{\sigma,R}$ denote the isometric
 isomorphisms given in
 Lemma \ref{lem:iso} and observe that
$\hat{J}=D_{\sigma,r}^{-1} J D_{\sigma,R}$ and
$P_n = D_{\sigma,r}^{-1} Q_{2n} D_{\sigma,R}$.
The bounds \eqref{eq:normJP2n} and \eqref{eq:normP2n}
now follow from Lemma~\ref{lem:J}.
\end{proof}

\begin{remark}
For practical purposes (see the algorithm in
Section~\ref{sec:applications}),
a more useful representation of the Lagrange--Chebyshev
interpolation operator $P_n$ is given by (see, for example, \cite{Elliott})
\begin{equation}\label{eq:Pn_practical}
(P_nf)(x) =  \frac{d_{0,n}(f)}{2} +\sum_{l=1}^{n-1} d_{l, n}(f)T_l(x)
\qquad (x\in [-1,1], f\in C([-1,1])),
\end{equation}
where
\[ d_{l,n}(f) = \frac{2}{n}\sum_{k=0}^{n-1} f(x_k) T_l(x_k) \qquad (f\in C([-1,1])),\]
which can be derived from expression \eqref{eq:Q2nf_practical}
and the fact that $Q_{2n} D_\sigma = D_\sigma P_n$.
\end{remark}

\subsection{Lagrange--Chebyshev
approximation of holomorphic functions on general elliptic domains}
\label{sec:lag_cheb_gen}

In the following, we generalise the results of the previous subsection
to functions holomorphic on domains bounded by ellipses with arbitrary
(distinct) foci $\gamma_+, \gamma_- \in \mathbb{C}$.

For $\gamma = (\gamma_+, \gamma_-)$, define a linear map $\alpha_\gamma\colon \mathbb{C} \to \mathbb{C} $
as
\[\alpha_\gamma(z) = \frac{\gamma_+ - \gamma_-}{2} z + \frac{\gamma_+ + \gamma_-}{2}.\]
Then \[E_{\gamma,\rho} = \alpha_\gamma (E_\rho)\]
is a domain bounded by an ellipse with foci at $\gamma_+$ and
$\gamma_-$. For $\gamma = (1, -1)$ we recover the standard elliptic
domain $E_\rho$.

Let $H^\infty(E_{\gamma,\rho})$ denote the Banach space of
bounded holomorphic functions on $E_{\gamma,\rho}$
equipped with the supremum norm. We have the following simple lemma.

\begin{lemma}\label{lem:isoCf}
Let $C_{\alpha_\gamma}$
be the composition
operator given by
\[C_{\alpha_\gamma}\colon f \mapsto f\circ \alpha_\gamma.\]
Then, $C_{\alpha_\gamma}$
is an isometric isomorphism between $H^\infty(E_{\gamma,\rho})$
and $H^\infty(E_{\rho})$.
\end{lemma}

\begin{proof}
As $a_\gamma$ is invertible, it follows that $C_{\alpha_\gamma}$ is invertible
with inverse $C^{-1}_{\alpha_\gamma} = C_{\alpha_\gamma^{-1}}$.
Also, for any $f\in H^\infty(E_{\gamma,\rho})$ we have
\[\norm{C_{\alpha_\gamma} f}{H^\infty(E_{\rho})} = \sup_{z\in E_\rho} |(f \circ \alpha_\gamma)(z)|
=\sup_{w \in \alpha_\gamma(E_\rho)} |f(w)| = \norm{f}{H^\infty(E_{\gamma,\rho})},\]
so $C_{\alpha_\gamma}$ is indeed an isometric isomorphism.
\end{proof}

Using the previous lemma we can now define generalised
Lagrange--Chebyshev projections $P_{\gamma,n}$ on
$H^\infty(E_{\gamma,R})$ by
\begin{equation}\label{eq:Pgn}
P_{\gamma,n} = C^{-1}_{\alpha_\gamma}
P_n  C_{\alpha_\gamma}.
\end{equation}
For $f$ a continuous function on the line segment
$\alpha_\gamma([-1,1])$,
$P_{\gamma,n}f$ is the polynomial of degree $n-1$ that coincides with $f$ at
the images of the Chebyshev points of order $n$ under
$\alpha_\gamma$. The following
is a straightforward consequence of Lemma~\ref{lem:Jhat} and the
previous lemma.

\begin{lemma}\label{lem:Jhatgamma}
Let $\hat{J_\gamma}\colon H^\infty(E_{\gamma,R}) \to H^{\infty}(E_{\gamma,r})$
denote the canonical embedding with $1<r<R$. Then, for any $n\in \N$,
\begin{equation}\label{eq:normJP2n_gamma}
  \|\hat{J_\gamma} - P_{\gamma,n}\|_{H^\infty(E_{\gamma,R})\to H^\infty(E_{\gamma,r})}
  \leq
c_{r,R}\, \frac{\cosh(n\log(r))}{\sinh(n\log(R))},
\end{equation}

\begin{equation}\label{eq:normP2n_gamma}
\norm{P_{\gamma,n}}{H^\infty(E_{\gamma,R})\to H^\infty(E_{\gamma,r})}\leq
c_{r,R} \, \frac{\cosh(n \log(R)) + \cosh(n\log(r))}{\sinh(n\log(R))},
\end{equation}
where
\[ c_{r,R} = \frac{\sinh(\log{R})}{\cosh(\log(R)) - \cosh(\log(r))}. \]
\end{lemma}

\section{Approximation of transfer operators with holomophic data}

In this section we consider general transfer operators associated with holomorphic
data on elliptic and annular domains. We show that Lagrange--Chebyshev
and equidistant Lagrange approximants converge to the respective original operators
at exponential speed
in operator norm. These results are obtained using a factorisation argument
already presented in various papers (see, for example, \cite{BJ_LMS, BJ_Advances, SBJNonlin, BJS}),
together with Lemmas \ref{lem:J} and \ref{lem:Jhatgamma}.

\subsection{Transfer operators with holomorphic data on elliptic
  domains}
\label{sec:trans_ellipse}

For the remainder of this section we shall fix $1<r<R$
and a pair of foci $\gamma = (\gamma_+, \gamma_-)\in \mathbb{C}^2$
with $\gamma_+\neq \gamma_-$.

\begin{defn}
Let $1<r<R$ and let $\I$ be a finite or countable index set.
A {\it holomorphic map-weight system}
is given by a family $(\Phi_i)_{i\in\I}$ of holomorphic maps in $H^\infty(E_{\gamma,R})$ satisfying
\begin{equation}\label{eq:cont}
\bigcup_{i\in \I} \Phi_i(E_{\gamma,R}) \subseteq E_{\gamma,r},
\end{equation}
and a family of weights $(W_i)_{i\in\I}$ in $H^\infty(E_{\gamma,R})$
satisfying
\begin{equation}
\label{eq:weights}
S_{\gamma,R} = \sup \left \{ \sum_{i\in \I} |W_i(z)| :
z\in E_{\gamma,R} \right \} < \infty,
\end{equation}
\end{defn}

With each such holomorphic map-weight system we associate a
{\it transfer operator} given by

\begin{equation}\label{eq:tr}
\tr f = \sum_{i\in\I} W_i \cdot f \circ \Phi_i.
\end{equation}

In the following, we shall see that the transfer operator maps
$H^\infty(E_{\gamma, r})$ compactly into itself; moreover, the transfer operator
also maps $H^\infty(E_{\gamma, R})$ compactly into itself, and we will use
the same symbol $\tr$ for the operator on $H^\infty(E_{\gamma, r})$ as well as for
its restriction to $H^\infty(E_{\gamma, R})$. We shall also see that the
transfer operator on either $H^\infty(E_{\gamma, r})$ or $H^\infty(E_{\gamma, R})$
can be effectively approximated using Lagrange--Chebyshev
interpolation, in a sense to be made precise below.

We start by proving that the transfer operator can be lifted to a
bounded operator
from $H^\infty(E_{\gamma,r})$ to $H^\infty(E_{\gamma,R})$.

\begin{lemma}\label{lem:Ltilde}
The transfer operator arising from a holomorphic map-weight system
maps $H^\infty(E_{\gamma,r})$ continuously to $H^\infty(E_{\gamma,R})$
with
\[ \| \tr \|_{H^\infty(E_{\gamma,r}) \to H^\infty(E_{\gamma,R})} \leq S_{\gamma,R}.\]
\end{lemma}

\begin{proof}
Without loss of generality we shall assume that $\I=\N$.
Fix $f\in H^\infty(E_{\gamma,r})$ and, for $k\in \N$, write
$g_k(z):=\sum_{i=1}^k W_i(z)f(\Phi_i(z))$. We clearly have
$g_k\in H^\infty(E_{\gamma,R})$. Since
\begin{equation}
\label{eq:trbdd}
|g_k(z)|\leq \sum_{i=1}^k |W_i(z)| |f(\Phi_i(z))|\leq
S_{\gamma,R}\norm{f}{H^\infty(E_{\gamma,r})}
\end{equation}
for all $z\in E_{\gamma,R}$, it follows that the sequence $( g_k )_{k\in \N}$ is uniformly
bounded on $E_{\gamma,R}$.  Moreover, the limit $\lim_{k\to\infty}g_k(z)=:g(z)$
exists for
every $z\in E_{\gamma,R}$.  By Vitali's convergence theorem (see, for example,
\cite[Proposition~7]{Nar}) the sequence $(g_k)_{k\in \N}$ thus
converges uniformly on compact subsets of $E_{\gamma,R}$.
Hence $g$ is analytic on $E_{\gamma,R}$.
Moreover, using (\ref{eq:trbdd}) it follows that
$|g(z)|\leq S_{\gamma,R} \|f\|_{H^\infty(E_{\gamma,r})}$
for any $z\in E_{\gamma,R}$.
Thus $\tr f=g\in H^\infty(E_{\gamma,R})$ and
$\norm{\tr f}{H^\infty(E_{\gamma, R})}\leq S_{\gamma,R}
\norm{f}{H^\infty(E_{\gamma, r})}$, as required.
\end{proof}

The lemma above implies that the transfer operator of a holomorphic
map-weight system is compact, when viewed as an operator from
$H^\infty(E_{\gamma,r})$ into itself. In order to see this, let
$\tilde{\tr}\colon H^\infty(E_{\gamma, r})\to H^\infty(E_{\gamma, R})$ denote the lifted transfer
operator which is bounded by Lemma~\ref{lem:Ltilde} and recall that the canonical
embedding $\hat{J}_\gamma \colon H^\infty(E_{\gamma,R})\to H^\infty(E_{\gamma,r})$ is
compact by Lemma~\ref{lem:Jhatgamma}. Thus $\tr=\hat{J}_\gamma\tilde{\tr}$
is a compact endomorphism of $H^\infty(E_{\gamma, r})$. The same argument shows
that $\tr=\tilde{\tr}\hat{J}_\gamma$ is a compact endomorphism of
$H^\infty(E_{\gamma, R})$. This factorisation argument is also at the heart of
the following theorem, our main result, which shows that
$\tr$ can be approximated at exponential speed using Lagrange--Chebyshev
projections.

\begin{theorem}\label{thm:ellipse}
Let $\tr$ be the transfer operator associated with a holomorphic
map-weight system and $(P_{\gamma, n})_{n\in\N}$ the sequence of
Lagrange--Chebyshev projections given in \eqref{eq:Pgn}. Then the following holds.
\begin{enumerate}[(i)]
\item
$\tr\colon H^\infty(E_{\gamma,r})\to H^\infty(E_{\gamma,r})$ is compact and
\[\norm{\tr - P_{\gamma,n} \tr}{\Hinf{E_{\gamma,r}}\to\Hinf{E_{\gamma,r}}} =
O\left(\left({\frac{r}{R}}\right)^n\right) \text{ as } n\to \infty\,. \]

\item
$\tr\colon H^\infty(E_{\gamma,R})\to H^\infty(E_{\gamma,R})$ is compact and
\[\norm{\tr - \tr P_{\gamma,n}}{\Hinf{E_{\gamma,R}}\to\Hinf{E_{\gamma,R}}} =
O\left(\left({\frac{r}{R}}\right)^n\right) \text{ as } n\to \infty\,. \]
\end{enumerate}
\end{theorem}

\begin{proof}
Let $\tilde{\tr} \colon H^\infty(E_{\gamma, r})\to H^\infty(E_{\gamma, R})$ denote the lifted
transfer operator from Lemma \ref{lem:Ltilde} and
$\hat{J}_{\gamma}$ the canonical embedding operator from
$H^\infty(E_{\gamma,R})$ to $H^\infty(E_{\gamma,r})$.
For part (i) we use the factorisation $\tr = \hat{J}_{\gamma} \tilde{\tr}$.
Lemmas \ref{lem:Ltilde} and \ref{lem:Jhatgamma} imply that
$\tr\colon H^\infty(E_{\gamma,r})\to H^\infty(E_{\gamma,r})$ is compact; moreover
\begin{align*}
\norm{\tr - P_{\gamma,n}\tr }{\Hinf{E_{\gamma,r}}\to\Hinf{E_{\gamma,r}}} &=
\norm{\hat{J}_{\gamma}\tilde{\tr} - P_{\gamma,n}\tilde{\tr}}{\Hinf{E_{\gamma,r}}
\to\Hinf{E_{\gamma,r}}}\\
&\leq
\norm{\hat{J}_{\gamma} - P_{\gamma,n}}{\Hinf{E_{\gamma,R}}\to\Hinf{E_{\gamma,r}}}
\norm{\tilde{\tr}}{\Hinf{E_{\gamma,r}}\to\Hinf{E_{\gamma,R}}}
\end{align*}
and the remaining assertion follows. For the proof of part (ii) write
$\tr = \tilde{\tr} \hat{J}_{\gamma}$ and proceed as in part~(i).
\end{proof}
\begin{remark}
A quick glance at the proof shows that the implied constants
in the theorem above can be made explicit.
\end{remark}

\subsection{Transfer operators with holomorphic data on annular domains}
\label{trans_annulus}

In this subsection we shall state results analogous to those in the previous subsection, but now
for transfer operators considered on the space of bounded holomorphic functions on
annular domains
containing the unit circle. The main application we have in mind is
to generalised transfer operators associated with analytic expanding circle maps, that is,
maps $\tau\colon \T\to\T$ with $\inf_{z\in\T}|\tau'(z)| > 1$. Note that any such map
admits
analytic extensions to annuli $A_R$ for suitable $R>1$.

\begin{defn}
\label{def:Lannular}
Let $\tau\colon \T\to\T$ be an analytic expanding circle map. Suppose that there are
positive real numbers  $r$ and $R$  with $1 < r < R$
such that both $\tau$ and $1/\tau$ are holomorphic on $A_R$, and such that
 \[\tau(A_r) \supset \cl{A_R},  \]
where $\cl{\cdot}$ denotes the closure of a set of $\mathbb{C}$.
To any such map $\tau$ and a weight function $w\in  H^{\infty}(A_R)$ we associate the
transfer operator
$\tr_\T\colon L^1(\T)\to L^1(\T{})$ defined by
\begin{equation}\label{eq:L_circle}
(\tr_\T f)(z) =
\sum_{\zeta:\,\tau(\zeta) = z} \frac{w(\zeta)}{|\tau'(\zeta)|} f(\zeta),
\end{equation}
where the summation extends over the (finitely many) pre-images of the point $z$ under
$\tau$.
\end{defn}

We shall now show that restricted to $H^{\infty}(A_R)$, this operator is compact and can
be approximated at exponential speed using equidistant interpolation.
In analogy with Lemma \ref{lem:Ltilde}, we first show that
$\tr_\T$ lifts to a bounded operator from $H^\infty(A_r)$ to $H^\infty(A_R)$.

\begin{lemma}\label{lem:L_circle_tilde}
The operator $\tr_\T$ given in Definition~\ref{def:Lannular}
maps $H^\infty(A_r)$ continuously to $H^\infty(A_R)$. In particular, we have
\[\norm{\tr_\T}{\Hinf{A_r}\to\Hinf{A_R}} \leq \norm{w}{H^\infty(A_r)}
\left(\frac{r^{\omega_\tau}}{\delta_+(r^{\omega_\tau})} +
\frac{r^{-\omega_\tau}}{\delta_-(r^{-\omega_\tau})}\right),\]
where $\delta_+(r) = \min_{|z|=r}|\tau(z)| -  R$ and
$\delta_-(r) =1/R -\max_{|z|=r}|\tau(z)|$, with $\omega_\tau = 1$ for orientation
preseving and $\omega_\tau = -1$ for orientation reversing $\tau$.
\end{lemma}
\begin{proof}
As $\tau$ is expanding, it is a $K$-fold covering for some $K>1$, and $|\tau'| =
\omega_\tau \tau'$.
Let $\phi_i$ denote the $i$-th local inverse branch of $\tau$, then $\tr_\T$ can
be written as
\[\tr_\T f = \omega_\tau\sum_{i=1}^{K} (w \circ \phi_i) \cdot \phi_i' \cdot (f \circ \phi_i).\]
Given $f\in H^\infty(A_r)$ with $\norm{f}{H^\infty(A_r)} \leq 1$, we want to show
that $\tr_\T f \in H^{\infty}(A_R)$. We follow the proof of \cite[Lemma 2.3]{SBJNonlin} by
estimating the asymptotic behaviour of Fourier coefficients of
$\tr_\T f$.
Using change of variables we can express the $n$-th Fourier coefficient
of $\tr_\T f$  as
\[c_n(\tr_\T f) = \frac{\omega_\tau}{2\pi i}\int_\T \frac{w(z) f(z)}{\tau(z)^{n+1}} \, dz.\]
Fix $n \geq 0$. By \cite[Theorem 17.11]{Rudin}, the (nontangential)
limit $f^*(z)$ for $z \in \partial A_r$ exists a.e.\ and $f^*$ is integrable on $\partial A_r$.
Moreover, as $z\mapsto \frac{w(z)}{\tau(z)^{n+1}}$ is holomorphic on $\cl{A_r}$ we may
deform the contour to obtain
\[|c_n(\tr_\T f)| = \left|\frac{\omega_\tau}{2\pi i}\int_{|z|=r^{\omega_\tau}}
\frac{w(z) f^\ast(z)}{\tau(z)^{n+1}} \, dz\right| \leq S_r
\frac{r^{\omega_\tau}}{\left(\inf_{|z|=r^{\omega_\tau}} |\tau(z)|\right)^{n+1}} =
S_r\frac{r^{\omega_\tau}}{\left(R+\delta_+(r^{\omega_\tau})\right)^{n+1}},
\]
where $S_r =\norm{w}{H^\infty(A_r)}$.
Similarly, for $n\geq 1$, we have
\[|c_{-n}(\tr_\T f)| \leq S_r r^{-{\omega_\tau}} \left(R^{-1} - \delta_-(r^{-{\omega_\tau}})
\right)^{n-1}.\]
As $\tau$ is holomorphic on $A_R$, it follows by the Open Mapping Theorem that
$\tau(A_r)$ is open.
Thus the condition $\tau(A_r) \supset \cl{A_R}$ implies $\delta_+(r^{\omega_\tau}) > 0$
and $\delta_-(r^{-\omega_\tau}) > 0$.
The same arguments as in the proof of \cite[Lemma 2.3]{SBJNonlin} now yield
$\tr_\T f \in H^\infty(A_R)$ as well as the claimed upper bound for the norm.
\end{proof}

Using the same factorisation argument already employed
in Section~\ref{sec:trans_ellipse},
we obtain the following analogue of Theorem~\ref{thm:ellipse}
for transfer operators on annular domains.

\begin{theorem}\label{thm:circle}
Let $\tr_\T$ be the transfer operator given in Definition~\ref{def:Lannular} and let
$(Q_{2n})_{n\in\mathbb{N}}$ be the sequence of equidistant Lagrange interpolation
projections given in \eqref{eq:Q2nf}. Then the following holds.
\begin{enumerate}[(i)]
\item $\tr_\T\colon H^{\infty}(A_r) \to H^{\infty}(A_r)$ is compact and
\[\norm{\tr_\T - Q_{2n} \tr_\T}{\Hinf{A_r}\to\Hinf{A_r}} =
O\left(\left({\frac{r}{R}}\right)^n\right) \text{ as } n\to \infty. \]
\item $\tr_\T\colon H^{\infty}(A_R) \to H^{\infty}(A_R)$ is compact and
\[\norm{\tr_\T - \tr_\T Q_{2n}}{\Hinf{A_R}\to\Hinf{A_R}} =
O\left(\left({\frac{r}{R}}\right)^n\right) \text{ as } n\to \infty. \]
\end{enumerate}
\end{theorem}

\begin{remark}
As before,
the implied constants in the theorem above can be made explicit.
\end{remark}

\subsection{Convergence of spectral data}
\label{sec:conv}
Theorems \ref{thm:ellipse} and \ref{thm:circle} together with standard results from spectral perturbation theory now imply the desired convergence of spectral data
of $P_{\gamma,n}\tr $ to spectral data of $\tr$, and similarly,
the convergence of spectral data of $Q_{2n}\tr_\T $ to spectral data of $\tr_\T$.
In order to avoid repetition, we shall state all results simultaneously
for both Lagrangre--Chebyshev and equidistant Lagrange interpolation.
For the remainder of this subsection, we write $U_\rho$, $P_n$, $\tr$ for either
$E_{\gamma,\rho}$, $P_{\gamma,n}$, $\tr$ (as in Theorem~\ref{thm:ellipse}) or
$A_{\rho}$, $Q_{2n}$, $\tr_\T$ (as in Theorem~\ref{thm:circle}), respectively, where in both cases $\rho$ may be either $r$ or $R$.

\begin{cor}\label{cor:conv}
Let $\tr\colon H^\infty(U_{r})\to H^\infty(U_{r})$ denote the transfer operator
as in Theorem~\ref{thm:ellipse} or  Theorem~\ref{thm:circle}
and let $\tr_n = P_{n}\tr$. Then the following holds.
\begin{enumerate}[(i)]
\item Any convergent sequence
$(\mu_n)_{n\in\mathbb{N}}$ with $\mu_n \in \operatorname{spec}(\tr_n)$
converges to a spectral point of $\tr$.

\item Conversely, for any $\mu \in \operatorname{spec}(\tr)$,
there exists a sequence $(\mu_n)_{n\in\mathbb{N}}$ with
$\mu_n \in \operatorname{spec}(\tr_n)$,
such that $\mu_n \to \mu$ as $n \to \infty$.
More precisely, if $\mu$ is an eigenvalue with
ascent\footnote{An eigenvalue $\mu$ of an
operator $T$ is said to have \emph{ascent} $\ell$,
if $\ell$ is the smallest integer such that the kernel of $(\mu I-T)^\ell$ equals
  that of $(\mu I-T)^{\ell+1}$. In particular, if $\mu$ is algebraically simple,
  then $\ell = 1$.} $\ell$,
we have
\[|\mu - \mu_n| =
O\left(\left(\frac{r}{R}\right)^{n/\ell}\right) \text{ as } n\to \infty.\]
\item
Let $\mu \in \operatorname{spec}(\tr)\setminus \{0\}$ and let
$(\mu_n)_{n\in\mathbb{N}}$ be a sequence with
$\mu_n \in \operatorname{spec}(\tr_n)$
such that $\mu_n \to \mu$ as $n \to \infty$.
Writing $\mathcal{P}$ for the spectral projection associated with the eigenvalue
$\mu$ of $\tr$ and $(h_n)_{n\in \N}$ for a sequence of
generalised eigenvectors associated with the eigenvalue $\mu_n$ of $\tr_n$,
normalised so that $\norm{h_n}{H^\infty(U_{r})}=1$, we have
\[\| \mathcal{P}h_n - h_n\|_{\Hinf{U_{r}}} = O\left(\left(\frac{r}{R}\right)^n\right) \text{ as } n\to \infty.\]
\end{enumerate}
\end{cor}

\begin{proof}
Statements (i) and (ii) are known as Properties U and L, respectively, and follow from
\cite[Corollaries 2.7, 2.13]{Ahues}, with the bound on the convergence rate following by
combining our Theorem \ref{thm:ellipse} (or Theorem \ref{thm:circle}, respectively) with Theorems 2.17, 2.18, and ensuing remarks
in \cite{Ahues}. Part (i) of Theorem \ref{thm:ellipse} (or Theorem \ref{thm:circle}, respectively)
and \cite[Proposition 2.9]{Ahues} finally yield
statement (iii).
\end{proof}

\begin{remark}
Analogous spectral approximation results hold for $\tr\colon H^\infty(U_{R}) \to
H^\infty(U_{R})$ and
$\tr_n=\tr P_{n}$. However, these are less important from a practical perspective. The
reason for this is that, while the non-zero eigenvalues of $P_{n}\tr$
and $\tr P_{n}$ coincide, this is not the case for the corresponding generalised
eigenspaces, and, as we shall see in Section~\ref{sec:applications}, the generalised
eigenvectors of $P_{n}\tr $ are easier to calculate than those of
$\tr P_n$.
\end{remark}

The corollary above provides estimates for the speed of convergence of the
generalised eigenvectors of $P_{n}\tr$ to the corresponding generalised eigenspace of $\tr$.
Similar results can
be obtained for the corresponding eigenfunctionals. In order to see this, let
$H^\infty(U_{R})^\ast$ denote the dual space of
$H^\infty(U_{R})$ equipped
with the usual strong dual topology turning $H^\infty(U_{R})^\ast$ into a
Banach space.

Part (ii) of Theorem~\ref{thm:ellipse} (or Theorem~\ref{thm:circle}, respectively)
implies that
\[\norm{\tr^\ast  - P_{n}^\ast \tr^\ast}{\Hinf{U_{R}}^\ast
\to\Hinf{U_{R}}^\ast} =
O\left(\left({\frac{r}{R}}\right)^n\right) \text{ as } n\to \infty, \]
that is, the adjoint approximant $P_{n}^\ast\tr^\ast$ converges to the
adjoint
$\tr^\ast$ of $\tr$ in the operator norm on $\Hinf{U_{R}}^\ast$, which in turn
implies convergence of the corresponding generalised eigenspaces. More precisely, we
have the following dual analogue of part (iii) of Corollary~\ref{cor:conv}.

\begin{cor}\label{cor:conv2}
Let $\tr^\ast \colon H^\infty(U_{R})^\ast\to H^\infty(U_{R})^\ast$
denote the adjoint transfer operator
as in Theorem~\ref{thm:ellipse} or Theorem~\ref{thm:circle}, and
let $\tr_n^\ast =P_{n}^\ast \tr^\ast$.
Suppose that $\mu \in \operatorname{spec}(\tr^\ast)\setminus \{0\}$ and that
$(\mu_n)_{n\in\mathbb{N}}$ is a sequence with
$\mu_n \in \operatorname{spec}(\tr_n^\ast)$
such that $\mu_n \to \mu$ as $n \to \infty$.
Writing $\mathcal{P^\ast}$ for the spectral projection associated with the eigenvalue
$\mu$ of $\tr^\ast$ and $(h_n^\ast)_{n\in \N}$ for a sequence of
generalised eigenvectors associated with the eigenvalue $\mu_n$ of $\tr_n^\ast$,
normalised so that $\norm{h_n^\ast}{H^\infty(U_{R})^\ast}=1$, we have
% $\mathcal{P}\varphi_n$ is non-zero for $n$ large enough and
\[\| \mathcal{P^\ast}h_n^\ast - h_n^\ast\|_{\Hinf{U_{R}}^\ast}
= O\left(\left(\frac{r}{R}\right)^n\right) \text{ as } n\to \infty.\]
\end{cor}

\section{Applications}\label{sec:applications}
In this section we apply the Lagrange approximation algorithm to several approximation
problems involving transfer operators and compare its performance with algorithms
available in the literature.

\subsection{Spectral data of Lagrange approximants}
Before going into details we briefly review how to calculate eigendata of
finite-rank approximants obtained by applying Lagrange projections to the transfer
operator of a holomorphic map-weight system. In order to avoid cluttered notation,
we write $X$ for any of the underlying Banach spaces discussed in the previous section,
$L \colon X\to X$ for the transfer operator arising from a holomorphic map-weight system, which
we merely need to assume bounded for the purpose of this discussion, and
$P \colon X\to X$  for a bounded projection of rank $n$, given by
\[ Pf=\sum_{l=1}^ne_l^\ast(f)e_l \qquad (f\in X),\]
where $e_l\in X$ and $e^\ast_l\in X^\ast$ for $l=1,\ldots, n$. Here $P$
may be thought of as any of the Lagrange projections given in
Equations~(\ref{eq:Q2nf_practical}), (\ref{eq:Pn_practical}) or
(\ref{eq:Pgn}).  Given $L$ and $P$ as above, define an $n\times n$ matrix $M$ by
\begin{equation}\label{eq:Mkl}
M_{kl} = e_k^\ast(Le_l) \qquad (k,l\in \{1,\ldots, n\}).
\end{equation}
It is not difficult to see, for example, by appealing to the Principle of Related Operators
(see \cite[Section~3.3]{Pietsch}) that the operators $PL$, $LP$, $PLP$ and $M$
have the same non-zero eigenvalues with the same algebraic multiplicities; in particular,
\[ \operatorname{spec}(PL)\setminus\{0\} =
\operatorname{spec}(LP)\setminus\{0\} =
\operatorname{spec}(PLP)\setminus\{0\} =
\operatorname{spec}(M)\setminus\{0\}.
\]
As for the corresponding generalised eigenvectors and eigenfunctionals, a short
calculation shows that for any non-zero $\mu \in \mathbb{C}$ and any $k\in \N$ we have
\begin{equation}
\label{eq:phiinker}
 h \in \ker ((\mu I - PLP)^k) \text{ iff }
 h=\sum_{l=1}^nx_le_l \text{ with } x \in \ker ((\mu I - M)^k),
\end{equation}
\begin{equation}
\label{eq:phistarinker}
h^\ast \in \ker ((\mu I - P^\ast L^\ast P^\ast)^k) \text{ iff }
h^\ast =\sum_{l=1}^nx_le_l^\ast \text{ with } x \in \ker ((\mu I - M^T)^k).
\end{equation}
Thus the generalised eigenvectors and eigenfunctionals of $PLP$ can be obtained from
the generalised right and left eigenvectors of $M$. Finally, we note that for any non-zero
$\mu \in \mathbb{C}$ and any $k\in \mathbb{N}$ we have
\begin{equation}
\label{eq:kerPLP}
\ker ((\mu I - PLP)^k) = \ker ((\mu I - PL)^k),
\end{equation}
so that the spectral convergence results contained in Corollaries~\ref{cor:conv} and
\ref{cor:conv2} can be applied to the
approximate eigendata obtained through (\ref{eq:phiinker}) and (\ref{eq:phistarinker}).

The proof of (\ref{eq:kerPLP}) is straightforward: given $h\in \ker(\mu I - PLP)$ we
have $\mu h = PLPh$, so $h$ must be in the range of $P$, hence $Ph=h$,
from which $\mu h = PLh$, that is, $h \in \ker (\mu I - PL)$. For the converse,
suppose that $h \in \ker (\mu I - PL)$, so $h$ must again be in the range of $P$,
from which $Ph=h$, and so $h \in \ker (\mu I - PLP)$. The proof of the general case with
$k>1$ is similar.

For the sake of completeness, we present a simple algorithm for the computation of matrix elements $M_{kl}$
in \eqref{eq:Mkl} for $L$ a transfer operator associated with a map-weight system holomorphic on standard ellipses, that is $\gamma = (-1, 1)$, and $P$ the Lagrange--Chebyshev interpolation operator in \eqref{eq:Pn_practical}.

\begin{algorithm}[H]\label{algo:lagcheb}
\caption{Lagrange--Chebyshev approximation of a transfer operator}
\begin{algorithmic}
\REQUIRE $n$, $\{W_j\}_{j=0,\ldots, d-1}, \{\Phi_j\}_{j=0,\ldots, d-1}$
\STATE
\COMMENT{Evaluation of data at interpolation points}
\FOR{$m=0$ to $n-1$}
\STATE $x_m \leftarrow \cos{\frac{\pi (2m+1)}{2n}}$
  \FOR{$j=0$ to $d-1$}
  \STATE $w_{jm} \leftarrow W_j(x_m)$
  \STATE $\phi_{jm} \leftarrow \Phi_j(x_m)$
  \ENDFOR
\ENDFOR\\

\STATE
\COMMENT{Computation of matrix entries}
\FOR{$k=0$ to $n-1$}
  \FOR{$l=0$ to $n-1$}
  \STATE $M_{kl} \leftarrow
  \frac{2-\delta_{0,k}}{n}\sum_{m=0}^{n-1} T_k(x_m) \sum_{j=0}^{d-1} w_{jm} T_l(\phi_{jm})$
  \ENDFOR
\ENDFOR
\RETURN M
\end{algorithmic}
\end{algorithm}

\begin{remark}
It it not difficult to see that $M = AB$ with
$A_{kl} = \frac{2-\delta_{0,k}}{n} T_k(x_l)$ and $B_{kl} = (L T_l)(x_k)$ for
$k,l=0,\ldots, n-1$, which yields
a more efficient way to compute $M$ in the algorithm above.
\end{remark}

\subsection{Decay of correlations}
In \cite{BB}, Bahsoun and Bose consider discretisation schemes based on piecewise
linear approximations to given functions and apply them to transfer operators
arising from interval maps with the transfer operator acting on the Banach space of
Lipschitz continuous functions. The resulting scheme is shown to provide
convergent approximations to the invariant densities together with rigorous error bounds
in a topology stronger than pointwise convergence (that is, stronger than the usual
$L^1$-convergence obtained through the standard Ulam method).

As an example, the authors establish rigorous error bounds for their
approximation scheme when applied to the full-branch map
$T\colon [0,1]\to[0,1]$ given by
\[
T(x)=
\begin{cases}
\frac{11x}{1-x} & 0 \leq x < \frac{1}{12}, \\
12x - i & \frac{i}{12} < x \leq \frac{i+1}{12},
\end{cases}
\]
where $i=1,\ldots, 11$.
Using the notation from Section \ref{sec:trans_ellipse} we obtain a holomorphic
map-weight system on an ellipse $E_{\gamma,R} \supset [0,1]$ with $\gamma=(0,1)$
and suitable $R>1$ (in fact, any $R \in (10, 20)$ is suitable) given by a family
$\{\Phi_i\}_{i=0}^{11}$ with $\Phi_0(x) = \frac{x}{11+x}$ and $\Phi_i(x) = \frac{x+i}{12}$ for $i=1,\ldots, 11$,
and a family of weights $\{W_i\}_{i=0}^{11}$ given by $W_i(x) = \Phi'_i(x)$.
The associated transfer operator given by
\eqref{eq:tr} is well-defined and compact on $H^\infty(E_{\gamma,R})$
with a simple leading eigenvalue at $1$. As is known, the rate of correlation decay
is determined by the subleading eigenvalue of $\tr$, which can easily
be approximated using the Lagrange--Chebyshev approximation algorithm.

In \cite{BB}, the subleading eigenvalue of $\lambda_2$ of $\tr$ was numerically
observed to be simple, which by Corollary \ref{cor:conv}(ii) implies
\[|\lambda_2 - \lambda_{2,n}|
= O \left(\left(\frac{r}{R}\right)^n\right) \text{ as } n \to \infty,\]
where $\lambda_{2,n}$ is the simple subleading
eigenvalue of $\tr_n$ (for sufficiently large $n$).
Below we compute $\lambda_{2,n}$ for different values of $n$ illustrating the rapid
exponential convergence of the scheme.
\vspace{0.1cm}
\begin{center}
\begin{footnotesize}
\begin{tabular}{ |c|c| }
 \hline
  2 & 0.0{\textcolor{mygray}{899609091606775605271181343464894043253988825186233560706}} \\
 12 & 0.090076127005295577{\textcolor{mygray}{7611934889786172935065120132357444867624}} \\
 22 & 0.0900761270052955778472464929999485{\textcolor{mygray}{481037667345626173002032}} \\
 32 & 0.09007612700529557784724649299994856269438059903552{\textcolor{mygray}{32965894}} \\
 42 & 0.0900761270052955778472464929999485626943805990355246068579 \\
 \hline
\end{tabular}
\end{footnotesize}
\end{center}
\vspace{0.1cm}
In order to obtain a good contraction ratio, we
performed a numerical search in the collection of ellipses
$E_{\gamma,R}$ with $\gamma=(0,1)$ and $R \in (1.01, 40)$. For each such ellipse, we
numerically determined the smallest ellipse $E_{\gamma,r}$
containing $\bigcup_i \Phi_i(E_{\gamma,R})$.
The best contraction ratio found in this way turned out to be
$(r/R)\approx 0.225$, which occurred for $R \approx 16.99$.
Comparing the resulting upper bound
with the empirically observed convergence rate $|\lambda_{2,n} - \lambda_{2,n-1}|$,
see Figure~\ref{fig:bahsoun}, shows that the theoretical bound is rather conservative in this case.

\begin{figure}[ht!]
  \centering
  \includegraphics[width=0.65\textwidth]{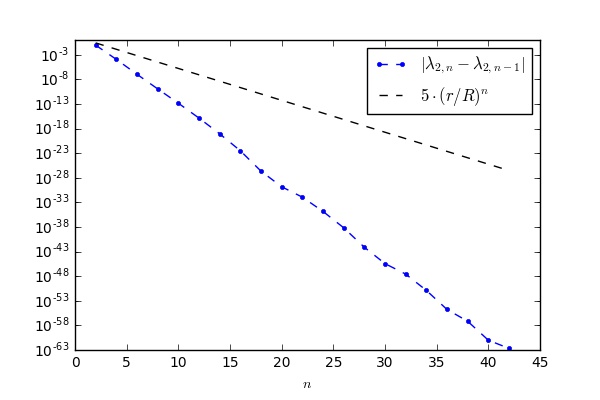}
  \caption{Absolute difference of subsequent approximations
  $\lambda_{2,n}$ as a function of $n$, computed using the Lagrange--Chebyshev
  algorithm (Algorithm $1$), compared to the upper bound given by the optimal
  contraction ratio in a family of confocal ellipses with foci $\gamma=(0,1)$.
  The scaling constant $5$ for the theoretical bound is chosen for visual
  reasons for both curves to intersect at $n=2$.}
  \label{fig:bahsoun}
\end{figure}

\subsection{Eigenvalues of transfer operators for circle maps}
In this subsection we shall approximate the spectrum of
transfer operators associated with expanding circle maps.
We shall choose Blaschke products as benchmark maps, as the corresponding
spectrum is available explicitly (see \cite{SBJNonlin, BJS}).
Using the same notation as in Section~\ref{trans_annulus},
let $\tau$ be a Blaschke product of degree two given by
\[\tau(z) = \left( \frac{z-\mu}{1-\bar{\mu} z} \right)^2, \quad |\mu| < 1. \]
For $|\mu| < 1/3$, the map $\tau$ yields an expanding circle map. We shall
approximate the spectrum of the transfer operator $\tr_\T$ in \eqref{eq:L_circle} with
weight function given by\footnote{The leading eigenfunctional of this transfer operator yields the measure of maximal entropy for $\tau$.} $w(z) = 1/\tau'(z)$. Using results from \cite{BJS}, it is not
difficult to see that the spectrum of $\tr_\T$ on $H^\infty(A_R)$ for a suitable annulus $A_R$ is given by
\[\operatorname{spec}(\tr_\T) = \{0, 2\} \cup
\{\left(\tau'(z_0)\right)^j,
\overline{\left(\tau'(z_0)\right)^j}\}_{j\in\N},\]
where $z_0$ is the unique attracting fixed point of $\tau$ in the open unit disk
$\mathbb{D}$.

For our numerical experiments we use $\mu = 0.33 \exp(i\pi/2)$ and
seek to approximate one of the subleading eigenvalues. Ordering the eigenvalues by decreasing modulus
\[ \lambda_1 = 2,\, \lambda_2 = \tau'(z_0),\, \lambda_3 = \overline{\tau'(z_0)},\,
\lambda_4 = \left(\tau'(z_0)\right)^2,\,\lambda_5 = \overline{\left(\tau'(z_0)\right)}^2, \,\ldots ,\]
we shall focus on the $7$-th eigenvalue in this sequence.
The fixed point $z_0$ of $\tau$ as well as $\tau'(z_0)$ are
available explicitly, however, as the expression is rather cumbersome,
we shall only give the numerical value of $\lambda_7$ using the first $60$ digits
\begin{align*}
\Re{(\lambda_7)} =\quad&
0.092670812973910200449109943460780953297549127162781618050227\ldots\\
\Im{(\lambda_7)} =-
&0.142165954484616119517212417389833629637614779964273017031665\ldots.
\end{align*}

Employing the equidistant Lagrange interpolation algorithm to calculate matrix
representations of $Q_{2n} \tr_\T Q_{2n}$ we obtain the following approximations of the
$7$-th eigenvalue of $\tr_\T$ for increasing values of $n$.
\vspace{0.1cm}

\begin{center}
\begin{footnotesize}
\begin{tabular}{ |c|c|c| }
 \hline
  {$n$}& $\Re{(\lambda_{7,n})}$ & $-\Im{(\lambda_{7,n})}$ \\
  23&
  0.0926708129739{\textcolor{mygray}{0991542848291151579381223686693}} & 0.142165954484616{\textcolor{mygray}{3819762201567962912837491514}} \\

  33&
  0.0926708129739102004491{\textcolor{mygray}{1051451873625154427548}} & 0.14216595448461611951721{\textcolor{mygray}{36940142097007716003}} \\

  43&
  0.0926708129739102004491099434607809{\textcolor{mygray}{3977781410}} & 0.1421659544846161195172124173898336{\textcolor{mygray}{025839907}} \\

  53&
  0.0926708129739102004491099434607809532975491{\textcolor{mygray}{0}} & 0.142165954484616119517212417389833629637614{\textcolor{mygray}{8}} \\
 \hline
\end{tabular}
\end{footnotesize}
\end{center}
\vspace{0.1cm}

Inspection of the table above shows that, for $n=53$,
the first $42$ decimal places of the approximate eigenvalue $\lambda_{7,n}$
coincide with the theoretical value $\lambda_7$. As a result, we see
that subleading complex eigenvalues close to $0$ of general transfer operators
can be approximated effectively using the equidistant Lagrange interpolation algorithm.

\subsection{Lyapunov exponents of random matrix products}
In this subsection we shall illustrate the use of the Lagrange--Chebyshev algorithm to
approximate Lyapunov exponents of random products of matrices. This is a rather
challenging task, in general. However, if all matrices are positive, it is possible to
use thermodynamic formalism to obtain the top Lyapunov exponent from a certain family
of transfer operators through periodic orbit expansions. This approach is originally due to
Pollicott~\cite{P2010}; more recently, an effective version has been proposed by Jurga
and Morris~\cite{JM}.
In the following we shall show how, instead of periodic orbit expansions,
the Lagrange--Chebyshev algorithm can be used effectively in this setup.

Let $\mathcal{A} = \{A_1, \ldots, A_K\}$ be a finite set
of positive invertible $2\times2$ matrices, let $(p_1,\ldots,p_K)$ be a
probability vector and denote by $\mathbb{P}_{p}$ the associated Bernoulli
measure on the space of sequences $\Omega = \{1, \ldots, K\}^{\mathbb{N}}$. The
(top) Lyapunov exponent of $\mathbb{P}_{p}$ is given by
\[\Lambda = \Lambda(\mathcal{A}, p) = \lim_{n\to \infty} \frac{1}{n}
 \int \log \| A_{\omega_1} \cdots A_{\omega_n}\|\, d\mathbb{P}_{p}(\omega),\]
 where $\omega\in\Omega$, and $\|\cdot\|$ denotes any matrix norm. Moreover, by \cite{FK} for $\mathbb{P}_{p}$-a.e.
 $\omega = (\omega_n)_{n\in\mathbb{N}}$ we have
\[\Lambda = \lim_{n\to \infty} \frac{1}{n} \log \| A_{\omega_1} \cdots A_{\omega_n}\|.\]

As was shown in \cite{P2010} and \cite{JM}, the Lyapunov exponent $\Lambda$ can be
expressed as the derivative of the top eigenvalue of $\tr_t$ with
respect to $t$ at $t=0$, where $\{\mathcal{L}_t\}_{t\in\mathbb{C}}$ is a certain
one-parameter family of transfer operators obtained as follows. To each matrix
\[ A = \begin{pmatrix}
a & b \\
c & d
\end{pmatrix}\in \mathcal{A}
\]
we associate a Moebius map ${\phi}_A$ given by
\[
{\phi}_A(z) = \frac{(a-b)z+b}{{w}_A(z)}
\]
 with ${w_A}(z) =(a+c-b-d)z +b+d$.
Writing $D_\rho(z_0)$ for a disk in $\mathbb{C}$ with radius $\rho$ and centre $z_0$, it
is not difficult to see that $\phi_A$ is a holomorphic self map of
$D_{\frac{1}{2}}(\frac{1}{2})$ with
$\phi_A(D_{\frac{1}{2}}(\frac{1}{2})) \subset D_\rho(\frac{1}{2})$
for suitable $\rho < 1/2$.
Moreover, one can show
that $\phi_A(E_{\gamma,R}) \subset E_{\gamma,r}$ for a suitable $1<r<R$ and
$\gamma = (\gamma_+, \gamma_-)$, with $\gamma_+, \gamma_-$ small (distinct)
perturbations of $1/2$.
We can now define a (so-called annealed) transfer operator $\tr$ on
$H^\infty(E_{\gamma,R})$ associated to $\mathcal{A}$ and
$p$ by
\[\tr_0 f = \sum_{i=1}^K p_i  f\circ \phi_{A_i}.\]
This operator is compact with leading eigenvalue $\lambda_0 = 1$, which turns out to be
simple \cite[Proposition 2.3]{JM} and the corresponding eigenfunction is the constant
function $h_0 = \mathbf{1}$.
We now define a family of perturbed transfer operators
\[\tr_t f = \sum_{i=1}^K p_i (w_{A_i})^{t} f\circ \phi_{A_i}
\qquad (t \in \mathbb{C}),\]
which are easily seen to be well-defined operators on $H^\infty(E_{\gamma,R})$.
This follows by observing that each weight function
$(w_{A_i})^t = \exp(t \log w_{A_i})$,
where $\log$ denotes the principal branch of the complex logarithm,
is a bounded holomorphic function
on $E_{\gamma,R}$ as $\Re{(w_{A_i}(z)}) > 0$ for $z\in E_{\gamma,R}$.
Since $t\mapsto \tr_t$ is a holomorphic family in $t$ and $\lambda_0$ is an
algebraically simple isolated eigenvalue of $\tr_0$, it follows by standard analytic
perturbation theory
(see, for example, \cite[Chapter~II, \S~1.8]{Kato})
that the largest (in modulus) eigenvalue
$\lambda_t$ of $\tr_t$ as well as the corresponding spectral projection
$\mathcal{P}_t$ are holomorphic in $t$ on an open neighbourhood around $t=0$.
We write $\mathcal{P}_tf=h^\ast_t(f)h_t$, where $h_t \in H^\infty(E_{\gamma,R})$ is the
eigenvector of $\tr_t$ corresponding to $\lambda_t$ and
$h_t^\ast \in H^\infty(E_{\gamma,R})^\ast$
the corresponding eigenfunctional, normalised so that $h_t^\ast(\mathbf{1})=1$.
By \cite[Proposition 3.1]{JM}, it follows that the Lyapunov exponent can
be expressed as the derivative of the top eigenvalue $\lambda_t$ at $t=0$, that is,
\[\Lambda = \frac{\partial\lambda_t}{\partial t}\Bigr\rvert_{t=0}.\]
As $\tr_t$, $h_t$ and $h^\ast_t$ are
holomorphic families in $t$ on a neighbourhood around $t=0$, a standard
computation (see, for example, \cite[Chapter~II, \S~2.2, Remark~2.2]{Kato}),
shows that
\[\Lambda =
h^\ast_0 ( \mathcal{M}_0 h_0), \]
where $\mathcal{M}_0$ denotes the
derivative of $\tr_t$ at $t=0$, which is given by
\[ \mathcal{M}_0f = \sum_{i=1}^K p_i \log(w_{A_i}) f\circ \phi_{A_i}. \]
The operator $\mathcal{M}_0$ is again a
well-defined operator on $H^\infty(E_{\gamma, R})$,
by the same arguments as before.

We shall now approximate the
Lyapunov exponent $\Lambda = h^\ast_0(\mathcal{M}_0 h_0)$
using Lagrange--Chebyshev approximation. For
$n\in\mathbb{N}$ let $\lambda_{0,n}$ denote the leading eigenvalue
of $\tr_{0,n} = \tr_0 P_{\gamma,n}$ and $h_{0,n}$ the corresponding leading
eigenfunction. Using (discrete) orthogonality properties  of Chebyshev polynomials
it is not difficult to see that
$\lambda_{0,n}=\lambda_0=1$ with $h_{0, n} = h_0 = \mathbf{1}$. The eigenfunctional
$h^\ast_{0,n}$ of $\tr_{0,n}$ corresponding to $\lambda_{0,n}=1$, normalised so that
$h^\ast_{0,n}(\mathbf{1})=1$, can be obtained numerically, as described in the opening
paragraphs of this section, and yields an approximation
\[ \Lambda_n=h^\ast_{0,n}(\mathcal{M}_0\mathbf{1}) \]
for the Lyapunov exponent $\Lambda$. In order to estimate the speed of convergence,
we note that
\[ \mathcal{P}_0^\ast h^\ast_{0,n}=h^\ast_0, \]
where $\mathcal{P}_0$ denotes the adjoint of the spectral projection $\mathcal{P}_0$ of
$\tr_0$ associated with $\lambda_0=1$. In order to see this observe that for any
$f\in H^\infty(E_{\gamma,R})$ we have
\[\mathcal{P}_0^\ast h^\ast_{0,n}(f)= h^\ast_{0,n}(\mathcal{P}_0f)=h^\ast_0(f)h^\ast_{0,n}
(\mathbf{1}) = h^\ast_0(f). \]

Thus, we have
\begin{align*}
|\Lambda- \Lambda_n|
= | h^\ast_0(\mathcal{M}_0\mathbf{1})-h^\ast_{0,n}(\mathcal{M}_0\mathbf{1}) |
&= |(\mathcal{P}_0^\ast h^\ast_{0,n}-h^\ast_{0,n})(\mathcal{M}_0\mathbf{1}) |
\\
&\leq \| \mathcal{P}_0^\ast h^\ast_{0,n}-h^\ast_{0,n} \|_{H^\infty(E_{\gamma,R})^\ast}
\| \mathcal{M}_0 \mathbf{1}\|_{H^\infty(E_{\gamma,R})},
\end{align*}
and so, by Corollary~\ref{cor:conv2},
\[|\Lambda - \Lambda_n|  = O \left(\left(\frac{r}{R}\right)^n\right).\]
Thus, the approximations converge exponentially in $n$, the size of the matrix
representing the Lagrange--Chebyshev approximant,
with the speed determined by the parameters of the ellipses $E_{\gamma,R}$ and
$E_{\gamma,r}$ satisfying
\[ \bigcup_i\phi_{A_i}(E_{\gamma,R} )\subset E_{\gamma,r}, \]
that is, complex contraction properties of the maps $\phi_{A_i}$.

We shall now test the performance of our algorithm using examples from \cite{JM}.
\begin{enumerate}
\item
  In Example 5.1 of \cite{JM},
  the matrices and probability vector are chosen to be
  \[ \mathcal{A} = \left \{
  \begin{pmatrix}
  2 & 1 \\
  1 & 1
  \end{pmatrix},
  \begin{pmatrix}
  3 & 1 \\
  2 & 1
  \end{pmatrix}
  \right \}
  \text{ and }
 p=(\frac{1}{2}, \frac{1}{2}).\]
 In this setting, the matrices
  in $\mathcal{A}$ strongly contract the positive quadrant, resulting in
  a highly effective approximation involving periodic orbits up to order $9$
  justifiably accurate to $31$ decimal places.
   Using our Lagrange--Chebyshev algorithm with $n=65$, we obtain the same
  reported value of
    \[\Lambda_n = 1.1433110351029492458432518536555882994025.\]
\item In Example 5.2 of \cite{JM}, the following choices are made
\[ \mathcal{A} =
\left \{
  \begin{pmatrix}
  3 & 1 \\
  1 & 3
  \end{pmatrix},
  \begin{pmatrix}
  5 & 2 \\
  2 & 5
  \end{pmatrix}
  \right \} \text{ and }
  p=(\frac{1}{2}, \frac{1}{2}).
  \]
 Note that in this case, we have
  $a_i + c_i = b_i + d_i$  so the weights $w_{A_i}$ are
  constant functions for $i=1,2$ hence $\mathcal{M}_0 \mathbf{1} = \mathbf{1}$,
  and therefore $\Lambda = \sum_{i=1}^{K}  p_i \log{(b_i + d_i)}$,
  which is the first entry of the matrix representation of $\tr_n$ for any $n$.
  In this case, the Lagrange--Chebyshev algorithm immediately yields the correct
  value, whereas the periodic orbit method from \cite{JM} only converges moderately  fast.
\item Perturbing the matrix entries of the previous trivial example, we now choose
\[ \mathcal{A} =
\left \{
  \begin{pmatrix}
  3.1 & 1 \\
  1 & 3
  \end{pmatrix},
  \begin{pmatrix}
  5.1 & 2 \\
  2 & 5
  \end{pmatrix}
\right \}
\text{ and }
  p=(\frac{1}{2}, \frac{1}{2}).
 \]
Applying the Lagrange--Chebyshev algorithm, we obtain very fast convergence in $n$,
as seen below
\begin{center}
\begin{footnotesize}
\begin{tabular}{ |c|c| }
 \hline
  1 & 1.67{\textcolor{mygray}{58722489713125213476722319655891806832879458231178998798}} \\
 10 & 1.67605018765901833052{\textcolor{mygray}{67823917604529846088532886001861874342}} \\
 20 & 1.6760501876590183305298001875390234510473{\textcolor{mygray}{825204275981348467}} \\
 30 & 1.6760501876590183305298001875390234510473137713601065642298 \\
 40 & 1.6760501876590183305298001875390234510473137713601065642298 \\
 \hline
\end{tabular}
\end{footnotesize}
\end{center}
\vspace{0.1cm}
In this case, the assumption $\Re{(w_{A_i}(z))} > 0$ for $i=1,2$ is satisfied
for $z$ with $\Re{(z)} \geq -40$, which in particular guarantees that
the weights $(w_{A_i})^t$ are holomorphic on $E_{\gamma,R}$ for $\gamma=(0,1)$
and $1 < R < \exp{(\text{arccosh}(40))} \approx 80$. Numerically, we may also
verify that this choice of ellipses satisfies
$\bigcup_i \phi_{A_i} (E_{\gamma,R}) \subset E_{\gamma,r}$ with $r < R$, yielding an
optimal contraction ratio of $(r/R) \approx 0.53$
for $R\approx 9.53$.
\end{enumerate}

\subsection{Approximation of stationary probability measures for iterated
function systems}
In \cite{CJ09}, Cipriano and Jurga study approximations of integrals with respect
to stationary probability measures associated to iterated function systems
on the interval, using an idea going back to Jenkinson and Pollicott \cite{JPSteklov}.
Their setting is an iterated function system $\{\Phi_i\}_{i = 1}^{K}$ consisting
of Lipschitz contractions $\Phi_i\colon [0,1] \to [0,1]$. Given
a probability vector $p = (p_1,\ldots, p_K)$, there exists a unique probability measure
$\nu$ such that
\[\int f\, d\nu = \sum_{i=1}^{K} p_i \int_{[0,1]} f \circ \Phi_i\, d\nu\]
for every continuous function $f\colon[0,1]\to \mathbb{R}$, see \cite{Hutch}. Assuming additionally that each $\Phi_i$ extends holomorphically to a neighbourhood of
$[0,1]$ and satisfies assumption \eqref{eq:cont} for a suitable ellipse $E_{\gamma,R} \supset [0,1] $,
we shall again consider the (annealed) transfer operator
\[\tr f = \sum_{i=1}^{K} p_i (f \circ \Phi_i),\]
which is a well-defined and compact operator on $H^\infty(E_{\gamma,R})$.
It has a simple eigenvalue $1$ with eigenfunction ${\bf 1}$, and the stationary
measure $\nu$ turns out to be the eigenfunctional $h^\ast$ of $\tr$ corresponding to the
eigenvalue $1$, normalised so that $h^\ast(\mathbf{1})=1$, that is
\[ h^\ast(f)=\int_{[0,1]}f\,d\nu. \]
Using the Lagrange--Chebyshev algorithm, we can obtain effective
approximations of the stationary measure through
eigenfunctionals of the approximants $\tr_n = \tr P_{\gamma, n}$.
In order to see this, note that $1$ is a simple eigenvalue of
$\tr_n$ and $h_n = \mathbf{1}$ is the corresponding eigenfunction of $\tr_n$. The
corresponding eigenfunctional $h^\ast_n$
can be obtained from the matrix representation of $\tr_n$, as explained in the introduction
to this section.

In \cite{CJ09}, the authors consider various integrals with respect
to the stationary measure arising from different iterated function systems, including an
application to the calculation of Lyapunov exponents. Here,
the Lyapunov exponent of the iterated function system $(\Phi, p)$ with respect to the
stationary measure $\nu$ is given by
\[\Lambda = - \int_{[0,1]} \sum_{i=1}^{K} p_i \log |\Phi_i'(x)|\, d\nu(x).\]
Suppose now that there is an $r$ with $1<r<R$ such that
\[ \bigcup_i\Phi_i(E_{\gamma, R})\subset E_{\gamma,r}. \]
Suppose also that the function
$x\mapsto \sum_i p_i \log | \Phi'_i(x) |$ has a
holomorphic extension $g$, which is bounded and holomorphic on the
elliptic domain $E_{\gamma, R}$.
Using the approximate eigenfunctional $h^\ast_n$ given above, we
obtain an approximation $\Lambda_n$ to the Lyapunov exponent $\Lambda$ by setting
\[ \Lambda_n= h^\ast_n(g).\]
In order to estimate the speed of convergence of this approximation we proceed as in the
previous subsection.  Let $\mathcal{P}$ denote the spectral projection
of $\tr$ associated with the leading simple eigenvalue~$1$, that is
\[ \mathcal{P}f=h^\ast(f)\mathbf{1}.\]
As before, we have $\mathcal{P}^\ast h^\ast_n = h^\ast$, hence
\begin{equation*}
|\Lambda - \Lambda_n| =
|h^\ast(g) - h^\ast_n(g)|
=|\mathcal{P}^\ast h^\ast_n(g)- h^\ast_n(g)| \\
\leq \| \mathcal{P}^\ast h^\ast_n -h_n^\ast\|_{H^\infty(E_{\gamma,R})^\ast}
\|g\|_{H^\infty(E_{\gamma,R})},
\end{equation*}
and Corollary~\ref{cor:conv2} now
yields
\[ |\Lambda - \Lambda_n|
= O\left(\left(\frac{r}{R}\right)^n\right).\]
We shall now compare our results to Example 6.8 in \cite{CJ09}. Let
\[
\Phi_1(x) = \frac{1}{6}\sin{(\pi x/4)} + \frac{1}{4},\qquad
\Phi_2(x) = \frac{1}{3}\sin{(\pi x/4)} + \frac{2}{3}, \qquad
p =(\frac{1}{3}, \frac{2}{3}).
\]
We may now approximate the corresponding Lyapunov exponent by $\Lambda_n$
for sufficiently large $n$, for example by
\[\Lambda_{100} = \quad
\scriptstyle{1.736720814737319877193356690960513773360205906006376079918873624791932498455557168841091},\]
where the first $89$ digits coincide with the value reported in \cite{CJ09}, which was
computed using multipliers of periodic orbits up to period $18$.
The optimal contraction
ratio of ellipses was numerically found to be $(r/R) \approx 0.4138$ in this case.

\section{Acknowledgement}
We gratefully acknowledge the support for the research presented in this
article by the EPSRC grant EP/RO12008/1. OFB would like to thank D\'onal MacKernan
for a first introduction to Chebyshev polynomials and their use in the study of expanding
interval maps.

\end{document}